\documentclass[12pt,a4paper]{article}

\usepackage[top=25mm, bottom=25mm, left=25mm, right=25mm]{geometry}

\usepackage{tikz} 
\usetikzlibrary{shapes.geometric,arrows,positioning}
 \usetikzlibrary{shapes.gates.logic.US,trees,positioning,arrows}
\usepackage{afterpage}
\usepackage{paralist}
\usetikzlibrary{arrows,calc,fit} 
\usepackage{lipsum}
\usepackage{rotating}
\usepackage{mathrsfs}
\usetikzlibrary{positioning,shadows,arrows}
\usepackage{enumerate}
\usepackage{ amsfonts,amsthm}
\usepackage{ multirow, graphicx,appendix}
\usepackage{amssymb}
\usepackage{amsbsy}

\usepackage{chemarr} 
\usepackage{enumitem} 
\usepackage{amsmath}

\usepackage{color}
\usepackage{amsmath,amsfonts,amssymb,amscd,amsthm,xspace}

\usetikzlibrary{patterns}
\usetikzlibrary{decorations.pathmorphing, shapes}
\usetikzlibrary{automata}

\usepackage{mathtools}

\title{A cellular algebra with specific decomposition of the unity }
\author{ Mufida M. Hmaida}
\date{}

\begin{document}

\newcommand{\Tn}{\mathbb{T}_{n,2}(\delta_0 , \delta_1 ) }
\newcommand{\TN}{\mathbb{T}_{n,2} }
\newcommand{\Tm}{\mathbb{T}_{n,m}(\delta_0 , \dots , \delta_{m-1} ) }
\newcommand{\TM}{\mathbb{T}_{n,m} }
\newcommand{\tn}{\mathcal{T}_{n,2} }
\newcommand{\tm}{\mathcal{T}_{n,m} }

\newcommand{\TL}{\mathsf{TL}_{n}( \delta) }
\newcommand{\Tl}{\mathsf{TL}_{n} }

\newcommand{\Pn}{\mathbb{P}_{n,2}(\delta_0 , \delta_1 ) }
\newcommand{\Pm}{\mathbb{P}_{n,m} (\delta_0 , \dots , \delta_{m-1} )}
\newcommand{\PN}{\mathbb{P}_{n,2} }
\newcommand{\pn}{\mathcal{P}_{n,2} }
\newcommand{\PM}{\mathbb{P}_{n,m} }
\newcommand{\ppm}{\mathcal{P}_{n,m} }

\newcommand{\Sym}{\mathfrak{S}}

\newcommand{ \Cc }{ \mathfrak{ C }}

\newcommand{\Pp}{\mathbb{P}_{n} ( \delta ) }
\newcommand{\Br}{\mathbb{B}_{n}( \delta) }

\newcommand{\Top}{\operatorname{top}}
\newcommand{\Bot}{\operatorname{bot}} 
\newcommand{\Hom}{\operatorname{Hom}}
\newcommand{\im}{\operatorname{im}}
\newcommand{\dom}{\operatorname{dom}}
\newcommand{\id}{\mathit{id}}
\newcommand{\ie}{\mathit{e}}
\newcommand{\Rad}{\operatorname{Rad}}
\newcommand{\rank}{\operatorname{rank}}
\newcommand{\Ker}{\operatorname{Ker}}
\newcommand{\type}{\operatorname{type}}

\newcommand{\A}{\mathbb{A}  }
\newcommand{\F}{\mathbb{F}  }
\newcommand{\Z}{\mathbb{Z}  }
\newcommand{ \Car }{ \mathbf{ \mathscr{ C } } }   
\newcommand{\Gr}{\mathsf{ G } }

\newcommand{\basis}{\mathsf{ C } }
\newcommand{\Basis}{\mathsf{ B } }
\newcommand{\V}{ V(\lambda ,i) }
\newcommand{\VL}{ \overline{V}(\lambda ,i) }

\theoremstyle{plain}
\newtheorem{theorem}{Theorem}[section]
\newtheorem{corollary}[theorem]{Corollary}
\newtheorem{lemma}[theorem]{Lemma}
\newtheorem{example}{Example}[theorem]
\newtheorem{proposition}[theorem]{Proposition}
\newtheorem{ass}[theorem]{Assumption}
\newtheorem{axiom}[theorem]{Axiom}

\newtheorem{definition}[theorem]{Definition}
\newtheorem{remark}[theorem]{Remark}

\maketitle

\textbf{Abstract.} Let $ \A $ be a cellular algebra over a field $\F$ with a decomposition of the identity $ 1_\A $ into orthogonal idempotents $\ie_i$, $i \in I$ (for some finite set $I$) satisfying some properties. We describe the entire Loewy structure of cell modules of the algebra $ \A $ by using the representation theory of the algebra $ \ie_i \A \ie_i $ for each $ i $. Moreover, we also study the block theory of $ \A $ by using this decomposition.

\section{Introduction}     

\hspace{15pt}
Let $ \Tm $, or simply $ \TM $, be the bubble algebra with $ m $-different colours, $ \delta_i \in \F $, which is defined in Grimm and Martin\cite{GM2}. In the same paper, it has been showed that it is semi-simple when none of the parameters $ \delta_i $ is a root of unity. Later,  Jegan\cite{JM} showed that the bubble algebra is a cellular algebra in the sense of Graham and Lehrer\cite{GL1}. The identity of the algebra $ \TM $ is the summation of all the different multi-colour partitions that their diagrams connect $j$ only to $j'$ with any colour for each $1 \leq j \leq n $, these multi-colour partitions are orthogonal idempotents. The goal of this paper is to generalize the technique that we use to study the representation theory of $ \TM $ in \cite{H}.

\hspace{15pt}
Wada\cite{W3} consider a decomposition of the unit element into orthogonal idempotents and a certain map $ \alpha $ to define a Levi type subalgebra and a parabolic subalgebra of the algebra $ \A $, and then the relation between the representation theory of each one has been studied. With using the same decomposition, we construct a Levi type subalgebra $ \bar{\A} $ (without using any map), and classify the blocks of $ \A $ by using the representation theory of the algebra  $ \bar{\A} $.

\section{Cellular algebras}   

\hspace{15pt}
We start by reviewing the definition of a \emph{cellular algebra}, which was introduced  by Graham and Lehrer\cite{GL1} over a ring but we replace it by a field, since we need this assumption later.

\begin{definition}\cite[Definition 1.1]{GL1}\textbf{.}
A cellular algebra over $ \F $ is an associative unital algebra $ \A $, together with a tuple $ ( \Lambda, T(.),  \basis, * ) $ such that 
\begin{enumerate}
\item The set $ \Lambda $ is finite and partially ordered by the relation $ \geq $.
\item For every $ \lambda \in \Lambda$, there is a non-empty finite set $ T(\lambda)$ such that for an pair $ (s,t) \in T(\lambda)\times T(\lambda)  $ we have an element $ c_{st}^\lambda \in \A $, and the set $ \basis := \{ c_{st}^{\lambda} \mid s,t \in T(\lambda) \text{ for some } \lambda \in \Lambda  \}$ forms a free $\F$-basis of $\A$. 
\item The map $* : \A \to \A$ is $\F$-linear involution (This means that $ * $ is an anti-automorphism with $ *^2=\id_{ \A}$ and $ (c_{st}^\lambda)^*=c_{ts}^\lambda$ for all $ \lambda \in \Lambda $, $ s,t \in T(\lambda) $).
\item For $ \lambda\in\Lambda, s,t \in T(\lambda)$ and $ a \in \A$  we have 
\begin{align}  \label{eq1.1}
 a c_{st}^\lambda \equiv \sum_{ u \in T(\lambda)} r_a^{(u,s)} c_{ut}^\lambda \mod \A^{ > \lambda}, 
 \end{align} 
where $ r_a^{(u,s)} \in \F$ depends only on $ a,u$ and $s$. Here $\A^{> \lambda}$  denotes the $ \F$-span of all basis elements with upper index strictly greater than $\lambda$.
\end{enumerate}
\end{definition}

\hspace{15pt}
For each $\lambda \in \Lambda$, the cell module $ \Delta( \lambda ) $ is the left $ \A $-module with an $ \F $-basis $ \Basis:= \{ c^{\lambda}_{s} \mid s \in T(\lambda) \} $ and an action defined by 
\begin{align} \label{eq1.2}
 a c^{\lambda}_{s} = \sum_{ u \in T(\lambda)} r_a^{(s,u)} c^{\lambda}_{u} \;\;\;\; (a \in \A , \, s \in T ( \lambda )), 
 \end{align}
 where $  r_a^{(s,u)} \in \F $ is the same coefficient that in \eqref{eq1.1} .

\hspace{15pt}
A bilinear form $ \langle \; \; , \; \; \rangle  : \Delta( \lambda )  \times \Delta( \lambda )  \rightarrow \F $ can be defined by 
\[  \langle c^{\lambda}_{s}, c^{\lambda}_{t} \rangle c^{\lambda}_{ub} \equiv c^{\lambda}_{us}c^{\lambda}_{tb} \mod \A^{ > \lambda} \;\;\;(s,t,u,b \in T(\lambda)).\]
Note that this definition does not depend on the choice of $ u,b \in T(\lambda) $.

\hspace{15pt}
 Let $ \Gr_\lambda $ be the Gram matrix for $ \Delta( \lambda ) $ of the previous bilinear form  with respect to the basis $ \Basis $. All Gram matrices of cell modules that will be mentioned in this paper are with respect to the basis $ \Basis $ with the bilinear form defined by \eqref{eq1.2}.

\hspace{15pt}
 Let $ \Lambda^0 $ be the subset $ \{ \lambda \in \Lambda \mid \langle  \; , \;  \rangle  \neq 0 \} $. The radical 
\[ \Rad ( \Delta(\lambda) ) = \{ x \in \Delta(\lambda) \mid \langle  x , y  \rangle=0 \text{ for any } y \in \Delta(\lambda) \}\]
 of the form $ \langle  \; , \; \rangle $ is an $ \A $-submodule of $  \Delta(\lambda) $.

\begin{theorem}\cite[Chapter 2]{M3}\textbf{.} \label{thm1:cell}
Let $ \A $ be a cellular algebra over a field $ \F $. Then
\begin{enumerate}
\item $ \A $ is semi-simple if and only if $ \det \Gr_\lambda \neq 0 $ for each $ \lambda \in \Lambda $.
\item The quotient module $  \Delta(\lambda) / \Rad ( \Delta(\lambda)) $ is either irreducible or zero. That means that  $ \Rad ( \Delta(\lambda) ) $ is the radical of the module $ \Delta(\lambda) $ if $ \langle \;  , \;  \rangle  \neq 0 $.
\item The set $ \{ L( \lambda) := \Delta(\lambda) / \Rad ( \Delta(\lambda)) \mid \lambda \in \Lambda^0 \} $ consists of all non-isomorphic irreducible $ \A $-modules.
\item Each cell module $ \Delta( \lambda ) $ of $\A$ has a composition series with sub-quotients isomorphic to $ L ( \mu) $, where $ \mu \in \Lambda^0 $. The multiplicity of $ L ( \mu ) $ is the same in any composition series of $ \Delta( \lambda) $ and we write $ d_{ \lambda \mu } = [  \Delta( \lambda ) : L( \mu) ] $ for this multiplicity.
 \item The \emph{decomposition matrix} $ \mathsf{D} = \big( d_{ \lambda \mu }  \big)_{  \lambda \in \Lambda ,  \mu \in \Lambda^0 } $ is upper uni-triangular, i.e. $ d_{ \lambda \mu}=0 $  unless $ \lambda \leq \mu $ and $ d_{ \lambda \lambda }=1 $ for $ \lambda \in \Lambda^0 $.
 \item If $ \Lambda $ is a finite set and $ \Car $ is the Cartan matrix of $ \A $, then $ \Car= \mathsf{D}^t \mathsf{D}  $.
\end{enumerate}
\end{theorem}

\section{A Levi type sub-algebra}

\hspace{15pt}
In this section, we construct a Live type subalgebra $ \bar{\A} $ of $ \A $ and study its representation theory.

The second and the third parts of the following assumption existed in Assumption~$ 2.1 $ in \cite{W3}. 

\begin{ass} \label{ass}
Throughout the remainder of this paper, we assume the following statements $(A1)- (A4)$.
\begin{enumerate}[label=(A\arabic*)] 
\item There exists a finite set $ I $. 
\item The unit element $1_\A$ of $\A$ is decomposed as $ 1_\A= \sum\limits_{i \in I} \ie_i $ with $ \ie_i \neq 0 $ and $ \ie_i \ie_j=0 $ for all $ i \neq j $ and $ \ie_i^2=\ie_i $. 
\item For each $ \lambda \in \Lambda $ and each $ t \in T ( \lambda )$, there exists an element $ i \in I $ such that 
\begin{align} \label{eq.1}
\ie_i c_{ts}^\lambda = c_{ts}^\lambda \;\;\;\;\; \text{for any } s \in T( \lambda ) . 
\end{align}
\item $ \ie_i^*=\ie_i $ for each $ i \in I $. Note that from \eqref{eq.1}, for each $ \lambda \in \Lambda $ and each $ t \in T ( \lambda )$ we have
\begin{align} \label{eq.2}
c_{st}^\lambda \ie_i = c_{st}^\lambda \ie_i^*= \big( \ie_i c_{ts}^\lambda \big)^* = \big(c_{ts}^\lambda \big)^*= c_{st}^\lambda \;\;\;\;\; \text{for any } s \in T( \lambda ) . 
\end{align}
\end{enumerate}
\end{ass}

\hspace{15pt}
From $ (A2) $ and $ (A3) $, we obtain the next lemma.

\begin{lemma} \cite[Lemma 2.2]{W3}\textbf{.} \label{thm.1}
Let $ t \in T(\lambda) $, where $ \lambda \in \Lambda $, and $ i\in I $ be such that $ \ie_i c_{ts}^\lambda = c_{ts}^\lambda $ for any $ s \in T( \lambda )  $. Then for any $ j \in I $ such that $ j \neq i $, we have $ \ie_j c_{ts}^\lambda =0$ for any $ s \in T( \lambda )  $. In particular, for each $ t \in T(\lambda) $, there exists a unique $ i\in I $ such that $ \ie_i c_{ts}^\lambda = c_{ts}^\lambda $ for any $ s \in T( \lambda )  $.
\end{lemma}

\hspace{15pt}
For $ \lambda \in \Lambda $ and $i \in I$, we define
\begin{align*}
T(\lambda, i)&= \{ t \in T(\lambda) \mid \ie_i c_{ts}^\lambda = c_{ts}^\lambda \text{ for any } s \in T( \lambda ) \}, \\
\Lambda_i&= \{ \lambda \in \Lambda \mid T(\lambda, i) \neq \emptyset \} ,\\
I_{\lambda} &= \{ i \in I \mid \lambda \in \Lambda_i \} .
\end{align*}

By Lemma \ref{thm.1}, we have 
\[ T ( \lambda) = \coprod_{i \in I} T(\lambda, i).\]
Note that $ \Lambda_i $ is a poset with the same order relation on $ \Lambda$ and $ \Lambda =\bigcup_{i \in I} \Lambda_i $. Moreover, $ \Lambda_i \neq \emptyset $ for each $ i \in I $, and that because of $ 0 \neq \ie_i \in \A $ and $ \ie_i^2=\ie_i $.

From $(A3)$ and Lemma \ref{thm.1}, the element $ \ie_i $ can be written in the form 
\[\sum\limits_{ \substack{\lambda\in \Lambda, \\s,t \in T(\lambda,i)}} b_{(s,t,\lambda)} c_{st}^\lambda  \]
 where $ b_{(s,t,\lambda)} \in \F $.

\begin{theorem} \label{thm.2}
The algebra $ \ie_i \A \ie_i$ is a cellular algebra with a cellular basis $ \basis_i:=\{ c_{st}^{\lambda} \mid s,t \in T(\lambda, i) \text{ for some } \lambda \in \Lambda_i \}$ with respect to the poset $ \Lambda_i$ and the index set $ T(\lambda, i) $ for $ \lambda \in \Lambda_i  $, i.e. the following property holds;
\begin{enumerate}[label=(\arabic*)] 
\item An $ \F $-linear map $* : \ie_i \A \ie_i  \to \ie_i \A \ie_i $ defined by $c_{st}^\lambda \mapsto c_{ts}^\lambda$ for all $ c_{st}^\lambda \in \basis_i $ gives an  algebra anti-automorphism of $ \ie_i \A \ie_i $.
\item For any $ a \in \ie_i \A \ie_i$, $ c_{st}^\lambda \in \basis_i $, we have 
\begin{align*} 
 a c_{st}^\lambda \equiv \sum_{ u \in T(\lambda,i)} r_a^{(u,s)} c _{ut}^\lambda \mod ( \ie_i \A \ie_i)^{ > \lambda}, 
 \end{align*} 
where $ ( \ie_i \A \ie_i)^{ > \lambda} $ is an $ \F $-submodule of $\ie_i \A \ie_i$ spanned by $ \{ \mathrm{C}_{st}^{\lambda'} \mid s,t \in T(\lambda', i) $ for some $ \lambda' \in \Lambda_i $ such that $ \lambda' > \lambda \}$, and $ r_a^{(u,s)} $ does not depend on the choice of $ t \in T(\lambda,i) $.
\end{enumerate}
\end{theorem}
\begin{proof}
Since $\basis $ is a basis of  $ \A $, $ \ie_i a \ie_i=a $ for all $ a \in \ie_i \A \ie_i $ and 
\begin{equation*}
\ie_i  c_{st}^\lambda \ie_i = \begin{cases}
            c_{st}^\lambda   & \text{if }  s,t \in T(\lambda,i), \\
             0  & \text{otherwise, } 
       \end{cases}
\end{equation*}
so the set $\basis_i$ is a basis of the algebra $ \ie_i \A \ie_i $. The first part follows from the fact the map $ * $ on the algebra $ \A $ leaves $ \ie_i \A \ie_i $ invariant. For the second part, from \eqref{eq1.1} we have 
\[ a c_{st}^\lambda \equiv \sum_{ u \in T(\lambda)} r_a^{(u,s)} c _{ut}^\lambda \mod \A^{ > \lambda}, \]
where $ r_a^{(u,s)} \in \F$ depends only on $ a,u$ and $s$. But $ \ie_i a=a $, so
\[ a c_{st}^\lambda \equiv \sum_{ u \in T(\lambda)} r_a^{(u,s)} \ie_i c_{ut}^\lambda \mod \A^{ > \lambda}= \sum_{ u \in T(\lambda,i)} r_a^{(u,s)} c_{ut}^\lambda \mod \A^{ > \lambda}. \]
Also by using Lemma \ref{thm.1} we can show $ \ie_i \A^{ > \lambda} \ie_i = (\ie_i \A \ie_i)^{ > \lambda} $, we are done. Moreover, cell modules $ \V $ for the algebra $\ie_i \A \ie_i$ can defined as follows:
\[ \V := \ie_i \Delta(\lambda) \;\;\; ( \lambda \in \Lambda_i). \]
The set $ \Basis_i:=\{ c_s^\lambda \mid s \in T(\lambda,i)\} $ is a basis of the module $ \V $.
\end{proof}

\hspace{15pt}
Define the algebra $ \bar{\A} $ to be $ \sum_{i \in I} \ie_i \A \ie_i $ (which is the same as $  \bigoplus_{i \in I} \ie_i \A \ie_i $ since  $ \ie_i \ie_j=0 $ for all $ i \neq j $). The identity of the algebra $ \ie_i \A \ie_i $ is the idempotent $ \ie_i $, so $ \bar{\A} \hookrightarrow \A $. Moreover, the algebra $ \bar{\A} $ turns out to be cellular with cell modules:
\[ V(\lambda, i) = \ie_i \Delta(\lambda) \;\;\; ( \lambda \in \Lambda_i, i \in I). \]
We put $ V(\lambda, i) = \{ 0 \} $ in the case $ \lambda $ is not an element in $ \Lambda_i $.

\begin{lemma} \label{thm.3}
Let $ \lambda \in \Lambda $, then
\[ \Delta ( \lambda ) = \bigoplus_{ i\in I}  V(\lambda, i)    \]
as an $\bar{\A}$-module.
\end{lemma}
\begin{proof}
It comes directly from the fact that $ 1_\A= \sum\limits_{i \in I} \ie_i $ and $ \ie_i \ie_j=0 $ if $ i \neq j $.
\end{proof}

\section{Idempotent localization}

\hspace{15pt}
In this section we computethe radical and Gram matrix of each cell module of the algebra $ \A $ by using the ones of the algebra $ \bar{\A} $.

Let $ c_{us}^\lambda, c_{tv}^\lambda \in \basis $ where $ s \in T(\lambda,i) $ and $ t \in T(\lambda,j) $ for some $ i,j \in I $. If $ i \neq j $, then $ c_{us}^\lambda c_{tv}^\lambda=0 $ which means $ \langle c_s^\lambda, c_t^\lambda \rangle =0 $ in $ \Delta (\lambda) $. If $ i=j $, then
\[ c^{\lambda}_{us}c^{\lambda}_{tv} \equiv \langle c^{\lambda}_{s}, c^{\lambda}_{t} \rangle c^{\lambda}_{uv} \mod \A^{ > \lambda}.\]
Since $ u, v $ do not have a role here, we can assume $ u,v \in T(\lambda,i) $ and then 
\[ c^{\lambda}_{us}c^{\lambda}_{tv} \equiv \langle c^{\lambda}_{s}, c^{\lambda}_{t} \rangle c^{\lambda}_{uv} \mod (\ie_i \A \ie_i)^{ > \lambda}.\]
Hence the inner product $ \langle c_s^\lambda, c_t^\lambda \rangle $ in $ \Delta ( \lambda )$ and the inner product $ \langle c_s^\lambda, c_t^\lambda \rangle $ in $V(\lambda, i) $ have the same value. Let $ M(\lambda,i) $ be the Gram matrix of  this inner product on the module $ V(\lambda, i) $ with respect to the basis $ \Basis_i $, then 
\begin{align} \label{eq.3}
\Gr ( \lambda )= \bigoplus_{ i\in I_\lambda } M(\lambda,i).
\end{align}
We can show the previous result by using the facts $ \Basis = \coprod_{i \in I} \Basis_i $, $ \Basis_i\cap \Basis_j= \emptyset $ whenever $ i\neq j $ and $ \langle x,y\rangle=0 $ in $ \Delta(\lambda) $ whenever $ x \in \V $, $ y \in V(\lambda,j)$ where $ i\neq j $.

The previous equation show that $ \det \Gr ( \lambda ) \neq 0 $ if and only if $ \det M ( \lambda,i ) \neq 0 $ for each $ i \in I $ such that $ \lambda \in \Lambda_i $, then the following fact is straightforward.

\begin{theorem} \label{thm.33}
The algebra $ \A $ is semi-simple if and only if the algebra $ \ie_i \A \ie_i $ is semi-simple for each $ i $.
\end{theorem}
\begin{proof}
It comes directly from \eqref{eq.3} and from Theorem~\ref{thm1:cell}.
\end{proof}

\begin{lemma} \label{thm.4}
Let  $ \lambda \in \Lambda^0 $. The head of the module $ \Delta( \lambda ) $, denoted by $ L( \lambda ) $, satisfies the relation
\[ \dim L ( \lambda ) = \sum_{i \in I_\lambda } \dim \VL, \]
where $ \VL $ is the head of the $ \ie_i \A \ie_i $-module $ \V $. We put $ \dim \VL=0 $ if $ \lambda $ is not contained in $ \Lambda_i^0 $.
\end{lemma}
\begin{proof}
This follows from the fact that $ \dim L ( \lambda ) = \rank ( \Gr ( \lambda) )  $ as the algebra is over a field and $ \lambda \in  \Lambda ^0  $ and by using \eqref{eq.3}.
\end{proof}

\begin{theorem} \label{thm.5}
Let $ \lambda \in \Lambda$, then
\begin{align*}
\Rad( \Delta( \lambda )) \cong \bigoplus_{ i\in I_\lambda} \Rad( V(\lambda,i)) 
\end{align*}
as a vector space and 
\begin{align*}
\Rad( \Delta( \lambda )) \cong \sum\limits_{ i\in I_\lambda } \Rad( V(\lambda,i))
\end{align*}
as an $\A$-module.
\end{theorem}
\begin{proof}
First part comes directly from the fact that they have the same dimension:
\begin{align*}
\dim \Rad( \Delta( \lambda )) &= \dim \Delta( \lambda ))-\dim L(\lambda),\\
 &= \sum\limits_{i\in I} \dim V(\lambda,i)-\rank \big( \bigoplus_{ i\in I_\lambda} M(\lambda,i) \big) \\
  &=\sum\limits_{i\in I_\lambda} \big( \dim V(\lambda,i)-\rank M(\lambda,i) \big),\\
  &= \sum\limits_{i\in I_\lambda} \dim \Rad( V(\lambda,i)).
\end{align*}
Note that $ \V=\{0\} $ if $ i $ is not in $ I_\lambda $.

Next part is coming from the fact that the basis $ \Basis $ of the module  $ \Delta( \lambda ) $ equals $\amalg_{i\in I} \Basis_i $ and $ \Basis_i=\{ c_s^\lambda \mid s \in T(\lambda,i) \} $ is a basis the module $ \V $, also $ \langle c_s^\lambda , c_t^\lambda \rangle=0 $ whenever $ s \in T(\lambda,i) $ and $ t \in T(\lambda,j) $ such that $ i \neq j $. Let $ x \in \Rad( V(\lambda,i)) $ for some $i \in I_\lambda$, so $ \langle c_s^\lambda , x \rangle=0 $ for all $ s \in T(\lambda,i) $. Moreover, it is clear that $ \langle c_t^\lambda , x \rangle=0 $ for all $ t \in T(\lambda,j) $ where $ i \neq j $, then $ x \in  \Rad( \Delta( \lambda )) $. Thus
\[   \sum\limits_{ i\in I_\lambda} \Rad( V(\lambda,i)) \subseteq  \Rad ( \Delta_n( \lambda )  ) , \]
but both of them have the same dimension thus they are identical. 
\end{proof}

\begin{corollary} \label{thm.6}
Let $ \lambda \in \Lambda^0 $, then
\[ L(\lambda) \cong \sum\limits_{ i\in I_\lambda } \VL ,\]
as an $\A$-module.
\end{corollary}
\begin{proof}
As $ \V \cap V(\lambda,j) = \{ 0 \} $ whenever $ i \neq j $, so
\[ L( \lambda ) = \sum\limits_{ i\in I_\lambda} \frac{   \V }{  \Rad ( \V) } \cong \sum\limits_{ i\in I_\lambda} \VL.  \qedhere \]
\end{proof}

\section{The block decomposition of $\A$}

\hspace{15pt}
The aim of this section is to describe the blocks of the algebra $\A$ over a field $\F$ by studying the homomorphisms between cell modules of $ \A $.

We say $ \lambda \in \Lambda $ and $ \mu \in \Lambda^0 $ are cell-linked if $ d_{ \lambda \mu}\neq 0  $. A cell-block of $\A$ is an equivalence class of the equivalence relation on $ \Lambda$ generated by this cell-linkage. From Theorem \ref{thm1:cell}, each block of $ \A $ is an intersection of a cell-block with $\Lambda^0$, see \cite{GL1}. Thus, if there a non-zero homomorphism between $ \Delta(\lambda) $ and $ \Delta (\mu) $ where $ \lambda, \mu \in \Lambda^0 $, then they belong to the same block.

Let $ \theta : \Delta ( \lambda ) \rightarrow \Delta( \mu ) $ be a homomorphism defined by $ c_{s}^\lambda \mapsto \sum_{u \in T( \mu)} \alpha_u c_{u}^\mu $. Now if $ s \in T(\lambda,i) $ for some $ i \in I $, then $ u \in T(\mu,i) $ since $ \theta ( c_{s}^\lambda) = \theta ( \ie_i c_{s}^\lambda) = \sum_{u \in T( \mu)} \alpha_u \ie_i c_{u}^\mu $, so 
\[ \theta ( c_{s}^\lambda) =  \sum\limits_{u \in T( \mu,i)} \alpha_u c_{u}^\mu .\]
Hence the map $ \theta $ can be restricted to define a homomorphism
\[ \theta\downarrow_{\ie_i \A \ie_i} : V ( \lambda,i ) \rightarrow V( \mu,i) \]
Now if $ \theta \neq 0 $, then there is $ c_{s}^\lambda $ such that $ \theta ( c_{s}^\lambda) \neq 0 $. Assume that $ s \in T(\lambda,i) $ for some $ i $, then $ \theta\downarrow_{\ie_i \A \ie_i} \neq 0 $, which means that both the sets $ T(\lambda,i), T(\mu,i) $ don't equal the empty set.

\hspace{15pt}
Let $ \lambda, \mu \in \Lambda_i $ for some $ i $, and $ \tau : V ( \lambda,i ) \rightarrow V( \mu,i ) $ be a homomorphism $ \ie_i \A \ie_i $-modules. By extending the map $ \tau $, we obtain a homomorphism $ \tau \uparrow^{\A} : \Delta ( \lambda ) \rightarrow \Delta( \mu ) $. Thus 
\[ \Hom_{\A} ( \Delta ( \lambda ) , \Delta_n ( \mu))= \{ 0 \} \]
 if and only if 
\[  \Hom_{\ie_i \A \ie_i } (  \V , V( \mu, i) ) = \{ 0 \} \]
 for each $ i \in I $. From this fact, we obtain the next theorem.
 
 \begin{theorem} \label{thm.7}
 Let $ \Lambda = \Lambda^0 $. Two weights $ \lambda $ and $ \mu $ in $ \Lambda $ are in the same block of $ \A $ if and only if there exist $ \nu_0, \dots , \nu_r $ in $\Lambda$ such that all the following hold:
 \begin{enumerate}
 \item  $ \lambda $ and $ \nu_0 $ are in the same cell-block of $ \ie_i \A \ie_i $ for some $ i \in I $.
  \item  For each $ j=0,\dots, r-1 $, $ \nu_j $ and $ \nu_{j+1} $ are in the same cell-block of $ \ie_i \A \ie_i $ for some $ i \in I $.
   \item  $ \mu $ and $ \nu_r $ are in the same cell-block of $ \ie_i \A \ie_i $ for some $ i \in I. $
 \end{enumerate} 
 \end{theorem}$ \qedhere $

\section{Examples}

\hspace{15pt} 
In this section, we use some simple example to illustrate the facts that have been showed in the previous sections.

Let $ \A=M_{n\times n} (\F) $ be an $ n\times n$ matrix algebra over $\F$. This algebra is cellular with indexing set $ \Lambda=\{n\} $ and $ I=T(n)=\{1, \dots , n\} $. For each $ i,j \in T(n) $, we take $ c_{ij}^n=E_{ij} $ where $ E_{ij} $ is the matrix with $ 1 $ at the $ (i,j) $-entry and $ 0 $ elsewhere.  As we have $ 1_\A=\sum_{i\in I} E_{ii} $ and the elements $ E_{ii} $ satisfy all the assumptions in \ref{ass}, thus we can apply our results from the previous sections. Note that $ E_{ii} \A E_{ii} $ is isomorphic to $ \F $ for each $ i $, so $ \A $ is semi-simple see Theorem \ref{thm.33}.

For the second example, let $\A$ be the algebra which is given by the quiver 
\[ 1 \mathrel{ \mathop{ \rightleftarrows }^{ a_{12} }_{ a_{21} }  } 2 \]
 with the relation $ a_{12}a_{21}a_{12}= a_{21}a_{12}a_{21}=0 $. 
The algebra is spanned by the elements
\[e_1,e_2, a_{12},a_{21},a_{12}a_{21}, a_{21}a_{12},\]
where $ e_i $ is the path of length zero on the vertex $i$. As left module $\A$ is isomorphic to
\[ \F\langle e_1, a_{21}, a_{12}a_{21} \rangle \oplus \F\langle e_2, a_{12}, a_{21}a_{12} \rangle.\]

The algebra $ \A $ is a cellular algebra with anti-automorphism defined by $ a_{ij}^*=a_{ji} $ and $ \Lambda=\{ \lambda_0,\lambda_1,\lambda_2\} $ where $ \lambda_0>\lambda_1>\lambda_2 $ and
\begin{align*}
T(\lambda_0)=\{1\}, \;\; T(\lambda_1)=\{1,2\},T(\lambda_2)=\{2\}.
 \end{align*}
We define
\begin{align*}
c_{11}^{\lambda_0}=a_{12}a_{21}, \;\;  \left(\begin{array}{cc} c_{11}^{\lambda_1} & c_{1\,2}^{\lambda_1} \\
c_{2\,1}^{\lambda_1} & c_{2\,2}^{\lambda_1} \end{array} \right) = \left(\begin{array}{cc} e_1 & a_{12}  \\
a_{21} & a_{21} a_{12}\end{array} \right), \;\; c_{22}^{\lambda_2}=e_2.
 \end{align*}
The set $ \basis =\{c_{st}^{\lambda} \mid s,t \in T(\lambda) \text{ for some } \lambda \in \Lambda  \}$ is a cellular basis of $ \A$. Note that $ \lambda_0 $ is not in $ \Lambda^0 $ although $ \Delta (\lambda_0) $ is simple.

The identity $ 1_\A $ equals $ e_1+e_2$ and this decomposition satisfies all the conditions in Assumption \ref{ass}. Also we have
\begin{align*}
e_1 \A e_1= \F\langle e_1, a_{12}a_{21} \rangle \;\; \;\; \;\; \;\; \Lambda_1=\{ \lambda_0, \lambda_1\},\\
e_2 \A e_2= \F\langle e_2, a_{21}a_{12} \rangle \;\; \;\; \;\; \;\; \Lambda_2=\{ \lambda_1, \lambda_2\}.
 \end{align*}

Note that $ J= \F \langle a_{12} a_{21} \rangle $ is a nilpotent ideal of $ e_1 \A e_1 $ and $ J' = \F \langle a_{21}a_{12} \rangle $ is a nilpotent ideal of $ e_2 \A e_2 $, so $\A$ is not semi-simple, from Theorem \ref{thm.33}. Let $ \Basis=\{ e_1, a_{21} \} $ be a basis of the module $ \Delta(\lambda_1) $, so $ V(\lambda_1,1)=\F\langle e_1 \rangle  $ and $ V(\lambda_1,2)=\F\langle a_{21} \rangle  $. Also we have $ V(\lambda_2,1)=\{ 0\}  $ and $ V(\lambda_2,2)=\F\langle e_2 \rangle  $. It is easy to show that $ V(\lambda_2,2),  V(\lambda_1,2)$ are isomorphic as $ e_2 \A e_2 $, so they are cell-linked. From Theorem \ref{thm.7}, the modules $ \Delta(\lambda_1) $ and $ \Delta(\lambda_2) $ are in the same block of $\A$. Moreover, 
\begin{align*}
Rad ( \Delta (\lambda_1)) &=Rad ( \F\langle e_1, a_{21} \rangle) \cong  Rad ( V( \lambda_1,1)+ V( \lambda_1,2))\\
&= V( \lambda_1,2) \cong V( \lambda_2,2) = \Delta (\lambda_2).
 \end{align*}
 
\subsection{The multi-colour partition algebra}

 \hspace{15pt}
 For $n \in \mathbb{N}$, the symbol $ \mathcal{P}_n$ denotes the set of all partitions of the set $\underline{n} \cup \underline{n'}$, where $\underline{n}=\{1, \dots , n \}$ and $\underline{n'}=\{ 1' ,  \dots , n' \}$.

Each individual set partition can be represented by a graph, as it is described in \cite{M1}.  Any diagrams are regarded as the same diagram if they representing the same partition.

Now the composition $\beta \circ \alpha $ in $\mathcal{P}_n$, where $\alpha, \beta \in \mathcal{P}_n$, is the partition obtained by placing $ \alpha $ above $ \beta $, identifying the bottom vertices of $ \alpha $ with the top vertices of $ \beta $, and ignoring any connected components that are isolated from boundaries. This product on $\mathcal{P}_n$ is associative and well-defined up to equivalence.

A $(n_1, n_2)$-partition diagram  for any $ n_1, n_2 \in \mathbb{N}^+ $ is  a diagram representing a set partition of the set  $\underline{n_1} \cup \underline{n_2'} $ in the obvious way.

The product on $\mathcal{P}_n$ can be generalised to define a product of $(n,m)$-partition diagrams when it is defined. For example, see the following figure.
 
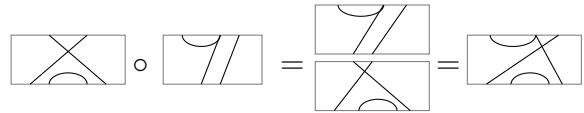
\begin{center}
 \begin{tikzpicture}
          \draw[gray]  (-4,0.35) rectangle (-2.5,1);
           \draw (-2.75,0.35)--(-3.5,1);
           \draw (-3.75,0.35)--(-3,1);                          
          \draw [domain=0:180] plot ({0.25*cos(\x)-3.25}, {0.15*sin(\x)+0.35 });
          
          \node[above] at (-2.3 ,0.35) {\sffamily $ \circ $};
          
          \draw[gray]  (-2,0.35) rectangle (-0.7,1);
           \draw (-1.25,0.35)--(-1,1);
           \draw (-1.5,0.35)--(-1.25,1);         
           \draw [ domain=360:180] plot ({0.25*cos(\x)-1.5}, {0.15*sin(\x)+1});                    
          
\node[above] at (-0.3,0.35) {\sffamily = };
          \draw[gray]    (0,0) rectangle (1.5,0.65);
   \draw[gray]  (0,0.75) rectangle (1.5,1.4);
          \draw (0.25,0)--(0.75, 0.65 ) ;
          \draw (0.5,0.65 )--(1.25,0) ;
          \draw [domain=0:180] plot ({0.25*cos(\x)+0.825}, {0.15*sin(\x)});
           \draw (0.5,0.75)--(0.9, 1.4 );
           \draw (0.75,0.75)--(1.2, 1.4 );
           \draw [ domain=360:180] plot ({0.3*cos(\x)+0.6}, {0.15*sin(\x)+ 1.4 });  

\node[above] at (1.75,0.35) {\sffamily = };

          \draw[gray]  (2,0.35) rectangle (3.5,1);

           \draw (2.25,0.35)--(3.2,1);
           \draw (3.25,0.35)--(2.9,1);         
           \draw [ domain=360:180] plot ({0.3*cos(\x)+2.6}, {0.15*sin(\x)+1});                    
          \draw [domain=0:180] plot ({0.25*cos(\x)+2.825}, {0.15*sin(\x)+0.35 });
\end{tikzpicture}
 \end{center}

 \hspace{15pt}
Let $ n,m $ be positive integers, $ \Cc_0 ,   \dots , \Cc_{m-1} $ be different colours where none of them is white, and  $ \delta_0,  \dots  ,  \delta_{m-1}  $ be scalars corresponding to these colours.

Define the set $ \Phi^{n,m} $ to be 
\[ \{ (A_0, \dots , A_{m-1}) \mid \{ A_0, \dots, A_{m-1} \} \in \mathcal{P}_n \}.\]

Let $ (A_0, \dots , A_{m-1}) \in \Phi^{n,m} $ (note that some of these subsets can be an empty set). Define $  \mathcal{P}_{A_0 , \ldots ,  A_{m-1}} $ to be the set $ \prod_{i=0}^{m-1} \mathcal{P}_{A_i}  $, where $ \mathcal{P}_{A_i}  $ is the set of all partitions of $ A_i$, and 
\begin{align*}
\ppm :=&  \bigcup\limits_{ (A_0, \dots , A_{m-1}) \in \Phi^{n,m} }  \mathcal{P}_{A_0 ,  \ldots ,  A_{m-1}} .
\end{align*}

The element $ d= (d_0, \dots , d_{m-1} ) \in \prod_{i=0}^{m-1} \mathcal{P}_{A_i} $ can be represented by the same diagram of the partition $ \cup_{i=0}^{m-1} d_i \in \mathcal{P}_{n} $ after colouring it as follows: we use the colour $ \Cc_{i} $ to draw all the edges and the nodes in the partition $ d_i $.

A diagram represents an element in $ \ppm $ is not unique.  We say two diagrams are equivalent if they represent the same tuple of partitions. The term multi-colour partition diagram will be used to mean an equivalence class of a given diagram. For example, the following diagrams
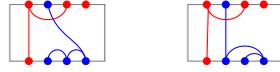
\begin{center}
\begin{tikzpicture}
          \draw[gray] (0,0) rectangle (1.25,0.75);
           \fill[red] (0.25,0)  circle[radius=1.5pt];
            \fill[red] (0.25,0.75)  circle[radius=1.5pt];
            \fill[blue] (0.5,0)  circle[radius=1.5pt];
            \fill[blue] (0.5, 0.75)  circle[radius=1.5pt];
            \fill[blue] (0.75,0)  circle[radius=1.5pt];
            \fill[red] (0.75, 0.75)  circle[radius=1.5pt];
            \fill[red] (1, 0.75)  circle[radius=1.5pt];
            \fill[blue] (1,0)  circle[radius=1.5pt];
            \draw[red] (0.25,0)--(0.25, 0.75 );
            \draw[blue] (0.5, 0.75) .. controls ( 0.6,0.3) and (0.95,0.3) .. (1, 0);
            \draw [domain=180:360][red] plot ({0.25*cos(\x)+0.5}, {0.2*sin(\x)+0.75});
            \draw[ domain=0:180][blue] plot ({0.125*cos(\x)+0.625}, {0.15*sin(\x)+0});
            \draw[ domain=0:180][blue] plot ({0.125*cos(\x)+0.875}, {0.15*sin(\x)+0});
\end{tikzpicture} 
\;\;\;\;\;\;\;
\begin{tikzpicture}
          \draw[gray] (0,0) rectangle (1.25,0.75);
           \fill[red] (0.25,0)  circle[radius=1.5pt];
            \fill[red] (0.25,0.75)  circle[radius=1.5pt];
            \fill[blue] (0.5,0)  circle[radius=1.5pt];
            \fill[blue] (0.5, 0.75)  circle[radius=1.5pt];
            \fill[blue] (0.75,0)  circle[radius=1.5pt];
            \fill[red] (0.75, 0.75)  circle[radius=1.5pt];
            \fill[red] (1, 0.75)  circle[radius=1.5pt];
            \fill[blue] (1,0)  circle[radius=1.5pt];
            \draw[red] (0.25,0)--(0.27, 0.75 );
            \draw[blue] (0.5, 0.75)--(0.5, 0);
            \draw [domain=180:360][red] plot ({0.25*cos(\x)+0.5}, {0.2*sin(\x)+0.75});
            \draw[ domain=0:180][blue] plot ({0.25*cos(\x)+0.75}, {0.2*sin(\x)+0});
            \draw[ domain=0:180][blue] plot ({0.125*cos(\x)+0.875}, {0.1*sin(\x)+0});
\end{tikzpicture} 
\end{center}
are equivalent.

 We define the following sets for each element $ d \in \prod \mathcal{P}_{A_i} $:
\begin{align*} 
 \Top(d_i)= A_i \cap\underline{n} , \;\;\;\;\;\;\;\;\;\;\;\;\;\;\;\;\;\;\;\;\;\;\;  \Bot(d_i) =  A_i \cap\underline{n}' , \;\;\;\;\;\;\;\;\;\;\;\;\;\;\;\;\;\;\;\;\;\;\;\; \\
 \Top(d)= ( \Top(d_0) ,\dots ,\Top(d_{m-1}) ), \;\;\;\;\;\;\; \Bot(d) =( \Bot(d_0) ,\dots ,\Bot(d_{m-1}) )  .
\end{align*}

Let $ \Pm $ be $\F$-vector space with the basis $ \ppm $, as it is defined in \cite{H}, and with the composition:
\begin{align*} 
( \alpha_i)( \beta_i) = \left\{ 
   \begin{array}{l l}
    \prod\limits_{i=0}^{m-1} \delta_i ^{c_i}  ( \beta_j \circ \alpha_j ) & \quad \text{if } \;  \Bot( \alpha)= \Top( \beta) ,\\
     0 & \quad \text{otherwise.}
   \end{array} \right. 
\end{align*}
where $\delta_i \in \F$, $ \alpha, \beta \in \ppm $, $c_{i}$ is the number of removed connected components from the middle row when computing the product $ \beta_i  \circ \alpha_i $ for each $ i=0, \dots , m-1 $ and $ \circ $ is the normal composition of partition diagrams.

The vector space $ \Pm$ is an associative algebra, called the \emph{multi-colour partition algebra}, with identity:
\[ 1_{ \PM } = \sum\limits_{(A_0, \dots , A_{m-1}) \in \Xi^{n,m} }  1_{(A_0 ,  \ldots ,  A_{m-1})} := \sum\limits_{(A_0, \dots , A_{m-1}) \in \Xi^{n,m} } ( 1_{A_0} ,  \ldots ,  1_{A_{m-1}} ) , \]
where $ \Xi^{n,m} := \{ (A_0 , \ldots ,  A_{m-1}) \mid \cup_{i=0}^{m-1} A_i= \underline{n}, A_i \cap A_j = \emptyset \; \forall i \neq j \} $, $1_{ A_i }$ is the partition of the set $ A_i \cup A'_i $ where any node $ j $ is only connected with the node $ j' $  for all $ j \in A_i$ and $ A_i'=\{ j' \mid j \in A_i \} $, for all $ 0 \leq i \leq m-1 $. This means the identity is the summation of all the different multi-colour partitions that their diagrams connect $i$ only to $i'$ with any colour for each $1 \leq i \leq n $.

\hspace{15pt}
The diagrams of shape $ \id \in \Sym_n $  are orthogonal idempotents, since
\[ 1_{(A_0 ,  \ldots ,  A_{m-1})} 1_{(B_0 ,  \ldots ,  B_{m-1})} =\left\{   \begin{array}{l l}
    0  & \quad \text{if }  \; (A_i) \neq (B_i) , \\
    1_{(A_0 ,  \ldots ,  A_{m-1})}   & \quad \text{if } \;  (A_i) = (B_i),
   \end{array} \right. \]
 for all $ (A_i), (B_i) \in   \Xi^{n,m} $. Thus we have a decomposition of the identity as a sum of orthogonal idempotents since  $ 1_{ \PM}=  \sum\limits_{(A_i) \in   \Xi^{n,m} } 1_{(A_0 ,  \ldots ,  A_{m-1})} $. As it have been showed in \cite{H}, the algebra $ \Pm$ is cellular and the last decomposition satisfies all the conditions in Assumption \ref{ass}.

Let $ (A_i) \in   \Xi^{n,m}$, then
\begin{align} 
1_{ (A_i)  } \Pm 1_{ (A_i)  } & \cong  \mathbb{P}_{ | A_0|} ( \delta_0)  \otimes _{\F} \cdots \otimes_{\F} \mathbb{P}_{ | A_{m-1}| } ( \delta_{m-1} ) ,  \label{eq2:1S1}
\end{align}
where $ \mathbb{P}_{ | A_i|} ( \delta_i) $ is the normal partition algebra and $ | A_i| $ is the cardinality of $ A_i $, for the proof see Chapter 2 in \cite{H}.

\begin{theorem}
The algebra $ \Pm $ is semisimple over $ \mathbb{C} $ for each integers $ n \geq 0 $ and $ m \geq 1 $ if and only if none of the parameters $ \delta_i $ is a a natural number less than $ 2n$.
\end{theorem}
\begin{proof}
As the algebra $ \Pp $ is semi-simple over $ \mathbb{C} $ whenever $ \delta  $ is not an integer in the range $ [0,2n-1] $, see Corollary 10.3 in \cite{MS3}, we obtain this theorem.
\end{proof}

\bibliographystyle{plain}

\end{document}